\documentclass[preprint,12pt]{elsarticle}
\usepackage[T1]{fontenc}   
\usepackage[utf8]{inputenc}
\usepackage{amsmath, amssymb, amsfonts, amsfonts, amsthm,latexsym}
\usepackage[mathscr]{euscript}
\usepackage{wasysym}
\usepackage{graphics}
\usepackage{comment}
\usepackage[normalem]{ulem}
\usepackage{graphicx}
\usepackage{color}
\usepackage[dvipsnames]{xcolor}
\usepackage{tikz}
\usepackage{marvosym}
\usetikzlibrary{arrows.meta, bending,shapes}

\newtheorem{theorem}{Theorem}[section]
\newtheorem{lemma}[theorem]{Lemma}
\newtheorem{corollary}[theorem]{Corollary}
\newtheorem{definition}{Definition}[section]
\newtheorem{proposition}[theorem]{Proposition}

\usepackage{algorithm}
\usepackage[noend]{algpseudocode}
\algrenewcommand\algorithmicrequire{\textbf{Input:}}
\algrenewcommand\algorithmicensure{\textbf{Output:}}

\newcommand{\dG}{\overrightarrow{G}}

\DeclareMathOperator{\lca}{lca}
\DeclareMathOperator{\parent}{parent}
\DeclareMathOperator{\child}{child}
\DeclareMathOperator{\cL}{cL}

\journal{Discrete Mathematics}

\begin{document}

\begin{frontmatter}

\title{Characterizations of undirected 2-quasi best match graphs}

\author[lpz]{Annachiara Korchmaros\corref{cor}} 
\ead{annachiara@bioinf.uni-leipzig.de}
\cortext[cor]{corresponding author}

\author[nwh]{Guillaume E. Scholz}

\affiliation[lpz]{organization={Bioinformatics Group, Department of Computer Science \& Interdisciplinary Center for Bioinformatics, Universit\"at Leipzig},
            addressline={H\"aertelstr. 16-18}, 
            city={Leipzig},
            postcode={D-04107}, 
            state={Sachsen},
            country={Germany}}

\affiliation[nwh]{organization={Independent},
            city={Leipzig},
            state={Sachsen},
            country={Germany}}

\author[lpz,MIS,TBI,BOG,SFI]{Peter F.\ Stadler}

\affiliation[MIS]{organization={Max Planck Institute for Mathematics in the
    Sciences}, addressline={Inselstra{\ss}e 22}, postcode={D-04103},
  city={Leipzig}, country={Germany}}

\affiliation[TBI]{organization={Institute for Theoretical Chemistry,
    University of Vienna}, addressline={W{\"a}hringerstrasse 17},
  postcode={A-1090}, city={Wien}, country={Austria}}

\affiliation[BOG]{organization={Facultad de Ciencias, Universidad Nacional
    de Colombia, Sede Bogot{\'a}}, country={Colombia}}

\affiliation[SFI]{organization={Santa Fe Institute}, addressline={1399 Hyde
    Park Rd.}, city={Santa Fe}, state={NM}, postcode={87501}, country={USA}}

\begin{abstract}
  Bipartite best match graphs (BMG) and their generalizations arise in
  mathematical phylogenetics as combinatorial models describing
  evolutionary relationships among related genes in a pair of species. In
  this work, we characterize the class of \emph{undirected 2-quasi-BMGs}
  (un2qBMGs), which form a proper subclass of the $P_6$-free chordal
  bipartite graphs. We show that un2qBMGs are exactly the class of
  bipartite graphs free of $P_6$, $C_6$, and the eight-vertex Sunlet$_4$
  graph. Equivalently, a bipartite graph $G$ is un2qBMG if and only if
  every connected induced subgraph contains a ``heart-vertex'' which is
  adjacent to all the vertices of the opposite color. We further provide a
  $O(|V(G)|^3)$ algorithm for the recognition of un2qBMGs that, in the
  affirmative case, constructs a labeled rooted tree that ``explains'' $G$.
  Finally, since un2qBMGs coincide with the $(P_6,C_6)$-free bi-cographs,
  they can also be recognized in linear time.
\end{abstract}

\begin{keyword}
  chordal bipartite graph \sep
  forbidden induced subgraphs \sep
  sunlet graphs \sep
  recognition algorithm \sep
  mathematical phylogenetics
\end{keyword}

\end{frontmatter}

\section{Introduction}

Best match graphs (BMGs) provide a combinatorial abstraction of the
evolutionary relationship ``\emph{being a closest relative}''. BMGs
naturally arise in comparative genomics, where they describe the
relatedness of homologous genes in different species. Given a leaf-labeled
tree $(T,\sigma)$ with leaf-coloring $\sigma$ representing the species in
which a gene (leaf) resides, there is a directed edge $x \to y$ in the
associated BMG whenever $y$ is the \emph{best match} of $x$ among all
leaves of color $\sigma(y)$. More precisely, $y$ is a best match for
$\sigma(y)$ if $\lca(x,y)$ equals or is a descendant of $\lca(x,z)$ for any
other leaf $z$ of the same color as $y$.  The resulting vertex-colored
digraph BMG$(T,\sigma)$ reflects part of the ancestral structure of $T$. It
has been studied extensively from graph-theoretic and algorithmic
viewpoints~\cite{geiss2019best,schaller2021corrigendum,
  schaller2021best,schaller2021heuristic}; see Section~\ref{sec:pre} for
formal definitions. This approach has recently been used to develop
software for estimating phylogenies of trees and orthology relations from
sequence data~\cite{ramirez2024revolutionh,ramirez2025revolutionh}.

Quasi-best match graphs (qBMGs) have been introduced as a generalization of
BMGs~\cite{korchmaros2023quasi}. Similar to BMGs, they derive from
leaf-colored tree $(T,\sigma)$ together with a \emph{truncation map} $u_T$
that assigns to each leaf $x$ and color $s \neq \sigma(x)$ a vertex on the
path from the root to $x$; see Section~\ref{sec:pre} for formal
definitions. This map determines how far the search for the best matches
extends toward the root.  The resulting digraph qBMG$(T,\sigma,u_T)$ thus
encodes edges based on partial ancestry information rather than the
complete tree topology.  This generalization preserves many structural and
algorithmic properties of BMGs, including polynomial-time tree
reconstruction, while capturing a broader range of leaf-colored digraphs.
In contrast to BMGs, qBMGs form a hereditary family of graphs. In
biological terms, qBMGs provide a combinatorial framework for evolutionary
relationships that depend on a limited ancestral depth, such as when genes are so distant that evolutionary relatedness cannot be determined
unambiguously.

If $\sigma$ takes two colors, the corresponding \emph{2-quasi best match graphs} (2qBMGs) are of remarkable interest in structural graph theory, not only for the already mentioned hereditary property but also for their rich combinatorial behavior, such as large automorphism groups~\cite{Korchmaros2026}, and finite characterizing 
set of forbidden subgraphs~\cite[Theorem~4.4]{schaller2021complexity}. The 2-colored BMGs are characterized as the sink-free  2qBMGs~\cite{korchmaros2023quasi}.

The \emph{undirected underlying graphs} of 2qBMGs, denoted \emph{un2qBMGs},
were introduced in~\cite{korchmaros2025forbidden} as a bridge between the
directed and symmetric best-match relations. The subgraphs of a BMG
comprising all its symmetric edges (i.e.\ 2-cycles) are the reciprocal
BMGs, which are fundamental in the study of gene family
histories~\cite{hellmuth2020complexity}. From a graph-theoretic
perspective, it is natural to study the underlying undirected graphs of
BMGs and qBMGs, leveraging well-established results on undirected graphs to
gain new insights into 2qBMGs.

Another remarkable property for structural graph theory, shown in ~\cite{korchmaros2025forbidden}, is that un2qBMGs admit no paths
of length six ($P_6$-free) and are chordal bipartite, i.e., they admit no
cycles of length six or more. This places un2qBMGs among structured
subclasses of bipartite graphs, sharing properties with chordal graphs
while remaining non-trivial. Recognition and optimization problems for
$P_6$-free chordal bipartite graphs (also known as $(P_6,C_6)$-free
bipartite) have been studied extensively since the seminal work
of~\cite{fouquet1999bipartite}.  The particular structure of
$(P_6,C_6)$-free graphs was found to be compatible with specific vertex
decompositions involving $K \oplus S$ graphs, where $K \oplus S$ denotes a
graph whose vertex set can be partitioned into a biclique $K$ and an
independent set $S$. In~\cite{quaddoura2024bipartite}, it is shown that a
bipartite graph $G$ is $(P_6,C_6)$-free if and only if every connected
subgraph of $G$ has a $K\oplus S$ vertex decomposition.  This structural
property enables the recognition of $(P_6,C_6)$-free bipartite graphs in
linear time~\cite{Takaoka:23}.  Since connected un2qBMGs also admit an $K
\oplus S$ vertex decomposition, it is natural to ask whether un2qBMGs can
be recognized in polynomial time.
  
The present paper provides a complete structural and algorithmic
characterization of un2qBMGs. In Section~\ref{sec:LRTs}, we define the
notion of a tree $(T,\sigma,u)$ that \emph{explains} a bipartite graph $G$
and show that $G$ is a un2qBMG if and only if $(T,\sigma,u)$ explains $G$.
We establish that every connected un2qBMG admits a \emph{least-resolved}
explaining tree, analogous to least-resolved trees of
BMGs~\cite{geiss2019best} and qBMGs~\cite{korchmaros2023quasi}. In
particular, in every least-resolved tree $(T,\sigma,u)$, each internal
vertex $v$ corresponds to an edge $xy$ of $G$ with $\lca_T(x,y)=v$, and
every associated subtree $T(v)$ is non-monochromatic.

In Section~\ref{sec:5vertices}, we then analyze the case of small instances
and bicliques, showing that all connected bipartite graphs with at most
five vertices are un2qBMGs; see Proposition~\ref{prop:small_cases}. In
Section~\ref{sec:algorithm}, we introduce the \textsc{HEART-TREE}
algorithm, which iteratively identifies \emph{heart-vertices}, i.e.,
vertices adjacent to all vertices of the opposite color, and uses them to
reconstruct a least-resolved explaining tree. Theorem~\ref{th:algoworks} shows that \textsc{HEART-TREE} decides in
cubic time whether a bipartite graph is a un2qBMG, and, if so, returns an
associated least-resolved tree. The key structural insight is that every
connected un2qBMG contains a heart-vertex, see Lemma~\ref{lm:vheart}, and that
the class of un2qBMGs is hereditary with respect to induced
subgraphs~\cite{korchmaros2025forbidden}.  Hence, one of the
characterizations of un2qBMGs in Theorem~\ref{th:5eq} is that a bipartite
graph is a un2qBMG if and only if every connected induced subgraph contains
a heart-vertex.

In~\cite[Theorem 3.4]{korchmaros2025forbidden}, the authors showed that
un2qBMGs are strictly contained in the class of $(P_6,C_6)$-free bipartite
graphs by providing a third forbidden subgraph $H$ for the class of
un2qBMGs. However, forbidden subgraphs contained in $H$ were not
excluded. In Section~\ref{sec:forbidden}, we study this problem and
characterize un2qBMGs as the bipartite graphs free of induced $P_6$, $C_6$,
and $\mathrm{Sunlet}_4$ subgraphs; see Theorem~\ref{th:5eq}.

\section{Preliminaries}\label{sec:pre}

\paragraph{Undirected and directed graphs}
Our notation and terminology for undirected graphs are standard.  Let $G$
be an undirected graph with vertex-set $V(G)$ and edge-set $E(G)$, and let
$\overrightarrow{G}$ be a directed graph with vertex-set $V(\dG)$ and
edge-set $E(\dG)$. For a graph $G$, the undirected edge $x-y$ is denoted as
$xy$, and the \emph{degree} of a vertex $x$ in $G$ is the number of
vertices $y \in V(G)$ such that $xy \in E(G)$. Similarly, the directed edge
$x \rightarrow y$ with head $x$ and tail $y$ is denoted as $xy$, and the
indegree and outdegree of a vertex $x \in V(\dG)$ is the number of vertices
$y \in V(\dG)$ such that $yx \in E(\dG)$ and $xy \in E(\dG)$,
respectively. In an undirected graph $G$, the \emph{degree} of a vertex
$x$, denoted $\deg_G(x)$, is the number of elements $y \in X$ such that $xy
\in E(G)$.

For an undirected graph $G$ and a subset $Y\subseteq V(G)$, the
\emph{subgraph of $G$ induced by $Y$}, denoted by $G[Y]$, is the graph with
vertex-set $Y$ and edge-set the set of edges of $G$ with both endpoints in $Y$.

A \emph{bipartite graph} $G$ is a graph whose vertices are partitioned into
two subsets (partition sets) such that the endpoints of every edge fall
into different subsets.  A \emph{bipartite digraph} $\dG$ if its underlying undirected graph is bipartite. The partition sets
may be viewed as the color classes of a proper vertex coloring $\sigma$ of
two colors in the sense that no edge is incident on vertices colored the
same color.

A path-graph $P_n$ with a path $v_1v_2\cdots v_{n}$ is a bipartite graph
with (uniquely determined) color sets $\{v_1,v_3,\ldots\}$ and
$\{v_2,v_4,\ldots\}$. An undirected graph is $P_n$-\emph{free} if it has no
induced subgraph isomorphic to a $P_n$ path-graph.  A \emph{cycle} $C_n$ of
an undirected graph $G$ is a sequence $v_1v_2\cdots v_n$ of pairwise
distinct vertices of $G$ such that $v_iv_{i+1}\in E(G)$ for
$i=1,\ldots,n-1$, and $v_nv_1\in E(G)$.  A $C_n$ \emph{cycle-graph} is an
undirected graph on $n$ vertices with a cycle $v_1v_2\cdots v_nv_1$
containing no chord, i.e. any edge in $E(G)$ is either $v_iv_{i+1}\in E(G)$
for some $1\le i \le n-1$, or $v_nv_1$. For $n$ even, a $C_n$ cycle-graph
with a cycle $v_1v_2\cdots v_nv_1$ containing no chord is a bipartite graph
with (uniquely determined) color classes $\{v_1,v_3,\ldots\}$ and
$\{v_2,v_4, \ldots\}$. Bipartite cycle-graphs can only exist for even
$n$. An undirected graph is $C_n$-\emph{free} if it has no induced subgraph
isomorphic to an $C_n$ cycle-graph. The \emph{sunlet graph}
$\text{Sunlet}_n$, for $n \geq 3$, is the graph obtained by attaching a
single pendant vertex to each vertex of the cycle-graph $C_n$. For $n=4$,
the sunlet graph is shown in Figure~\ref{fig:sunlet4}.
\begin{figure}[ht]
  \centering
\begin{tikzpicture}[x=1cm, y=1cm, rotate=45,
    fulldot/.style={circle, fill=black, inner sep=1.4pt},
    emptydot/.style={circle, draw=black, fill=white, inner sep=1.4pt}
  ]
  % Main cycle vertices with alternating full/empty dots
  \node[fulldot]   (v1) at (0:1)    {};
  \node[emptydot]  (v2) at (90:1)   {};
  \node[fulldot]   (v3) at (180:1)  {};
  \node[emptydot]  (v4) at (270:1)  {};
  
  % Pendant vertices with opposite styles
  \node[emptydot]  (u1) at (0:1.5)   {};
  \node[fulldot]   (u2) at (90:1.5)  {};
  \node[emptydot]  (u3) at (180:1.5) {};
  \node[fulldot]   (u4) at (270:1.5) {};
  
  % Edges of the cycle
  \draw (v1) -- (v2) -- (v3) -- (v4) -- (v1);
  
  % Pendant edges
  \draw (v1) -- (u1);
  \draw (v2) -- (u2);
  \draw (v3) -- (u3);
  \draw (v4) -- (u4);
\end{tikzpicture}
  \caption{$\text{Sunlet}_4$ with vertex bipartion induced by the
    vertex-coloring.}
  \label{fig:sunlet4}
\end{figure}
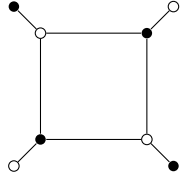

\paragraph{Rooted Trees}

A directed graph $T$ is a tree if all vertices of $T$ have indegree $0$ or
$1$. A tree $T$ is \emph{rooted} if exactly one vertex of $T$ has indegree
$0$. In this case, the vertex is the \emph{root} of $T$, denoted
$\rho_T$. In a rooted tree, its root is typically drawn as the top, and the
edges are directed from the parent vertices to their child vertices. A
vertex $v$ of $T$ is a \emph{leaf} if it has outdegree $0$, and we denote
the set of leaves of $T$ by $L(T)$. A vertex of $T$ that is not a leaf is
called an \emph{internal vertex} of $T$. Similarly, an arc $vw$ of $T$ such
that $w$ is an internal vertex of $T$ is called an \emph{internal arc} of
$T$.

A tree $T$ is \emph{phylogenetic} if every internal vertex has at least two
children. In this contribution, all trees are rooted and phylogenetic. Let
$v,w \in V(T)$. Whenever we write $vw \in E(T)$, we assume $w$ is a child
of $v$, and $v$ is the parent of $w$ denoted by $\parent_T(w)$ or, when
clear from context $\parent(w)$. We write $\child_T(v)$ or, when clear from
context, $\child(v)$ for the set of children of $v$ in $T$, and $\cL(v)$
for the set of children of $v$ that are leaves, that is, $cL(v)=\child_T(v)
\cap L(T)$.

\paragraph{Quasi-best match graphs}
Let $(T,\sigma)$ be a leaf-colored tree with leaf-coloring $\sigma:
L\rightarrow \sigma(L(T))$. For any two leaves $x,y\in V$, $\lca(x,y)$
denotes the last common ancestor between $x$ and $y$ on $T$.  A leaf $y \in
L$ is a best match of the leaf $x\in L$ if $\sigma(x)\ne\sigma(y)$ and
$\lca(x,y) \preceq \lca(x,z)$, i.e. $\lca(x,z)$ is an ancestor of
$\lca(x,y)$, holds for all leaves $z$ of color $\sigma(y) =\sigma(z)$. The
{\emph{BMG (best match graph) explained by}} $(T,\sigma)$ is the
vertex-colored digraph $BMG(T,\sigma)$ with color-set $\sigma(L(T))$,
vertex-set $L(T)$, and $xy\in E(G)$ if $y$ is a best match of $x$ in
$(T,\sigma)$.
\begin{figure}
    \centering
    \begin{tikzpicture}[>={Stealth[bend]},x=1cm,y=1cm,bullet/.style={fill,circle,inner sep=1.4pt}, scale=0.7]
\begin{scope}[nodes=bullet,scale=0.4]
   \node[label=below:{ $x_2$}, fill=black] (b) at (4,0) {};
   \node[label=below:{ $x_4$}, fill=black] (d) at (12,0) {};
   \node[label=below:{ $x_3$}, fill=white,draw=black] (c) at (8,0) {};
   \node[label=below:{ $x_1$} , fill=white,draw=black] (a) at (0,0) {};
   \node[label=left:{ $\lca(x_1,x_2)$}] (e) at (2,4) {};
   \node[label=right:{ $\lca(x_3,x_4)$}] (f) at (10,4) {};
   \node[label=above:{ $\rho$}] (h) at (6,8) {};
 \end{scope}
\begin{scope}[ every edge/.style={draw=black},scale=0.3]
   \path [-, thick] 
    (a) edge[bend left=0] (e)
    (b) edge[bend left=0] (e)
    (c) edge[bend left=0] (f)
    (d) edge[bend left=0] (f)
    (e) edge[bend left=0] (h)
    (f) edge[bend left=0] (h); 
\end{scope}
\end{tikzpicture}
%2qBMG
\hspace{0.5cm}
\begin{tikzpicture}[>={Stealth[bend]},x=1cm,y=1cm,bullet/.style={fill,circle,inner sep=1.4pt}]
\begin{scope}[nodes=bullet,scale=0.4]
   \node[label=below:{ $x_4$}, fill=black] (b) at (4,0) {};
   \node[label=above:{ $x_1$}, fill=white, draw=black] (c) at (0,4) {};
   \node[label=below:{ $x_2$}, fill=black] (a) at (0,0) {};
   \node[label=above:{ $x_3$} , fill=white, draw=black] (d) at (4,4) {};
 \end{scope}
\begin{scope}[ every edge/.style={draw=black},scale=0.5]
\draw [-] (c) edge [->, thick] (a);
\draw [<-,gray,dashed,thick] (c.south) to [out=330,in=30] (a.north);
\draw [-] (d) edge [<->, thick] (b);
\end{scope}
\end{tikzpicture}
    \caption{Example of a 2qBMG and 2BMG explained by a rooted phylogenetic
      tree. Black edges represent the edge-set of the 2qBMG. Black and
      dashed edges form the edge-set of the 2BMG. The truncation map $u$ of
      the 2qBMG is $u(x_2)=x_2$ and $u(x)=\rho$ otherwise.}
    \label{fig:2qBMG_tree}
\end{figure}
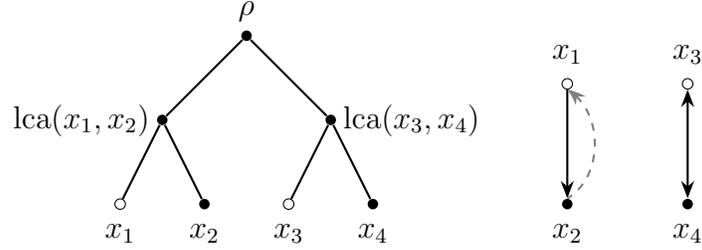
A truncation map $u_T: L(T)\times \sigma(L(T)) \rightarrow V(T)$ assigns to
every leaf $x \in L(T)$ and color $s \in \sigma(L(T))$ a vertex of $T$ such
that $u_T(x,\sigma(x))=x$ and $u_T(x,s)$ lies along the unique path from
$\rho_T$ to $x$ for $s\neq \sigma(x)$.  A leaf $y\in L(T)$ is a quasi-best
match for $x \in L(T)$ with respect to $(T,\sigma,u_T)$ if (i) $y\,$ is a
best match of $x$ in $(T,\sigma)$, and (ii) $\lca_T(x,y) \preceq
u_T(x,\sigma(y))$. The {\emph{qBMG (quasi-best match graph) explained by}}
$(T,\sigma,u_T)$ is the vertex-colored digraph $qBMG(T,\sigma,u_T)$ with
color-set $\sigma(L(T))$, vertex-set $L(T)$, and $xy\in E(G)$ if $y$ is a
quasi-best match of $x$ in $(T,\sigma,u_T)$. A vertex-colored digraph
$(\dG,\sigma)$ is a {\emph{$|S|$-colored quasi-best match
  graph}}($|S|$-qBMG) if there is a leaf-colored tree $(T,\sigma)$ together
with a truncation map $u_T$ on $(T,\sigma)$ such that $V(G)=L(T),
|S|=|\sigma(L(T))|,$ and $\dG= |S|$-qBMG($T,\sigma,u_T$).  In particular,
BMGs are qBMGs when $u_T(x,s)=\rho_T$ for $s\neq \sigma(x)$. For general
results on quasi-best match graphs, the reader is referred
to~\cite{korchmaros2023quasi}.

A leaf-colored tree $(T, \sigma,u )$ explaining a qBMG $\dG$ is
\emph{least-resolved} if there is no $(T,\sigma,u')$ explaining $\dG$ such
that $T'$ can be obtained from $T$ by a sequence of edge-contractions. BMGs
are explained by a unique least-resolved tree that can be built in
polynomial time~\cite{geiss2019best}. This property has recently been used
to build parsimonious explanations of a family of
genes~\cite{ramirez2024revolutionh}. On the other hand, qBMGs do not have a
unique least resolved tree, but a phylogenetic tree that explains a qBMG
can be built in polynomial time~\cite{korchmaros2023quasi}.

In this contribution, we are interested in 2qBMGs, that is, $|\sigma(L(T))|=2$. In this case, $u(x)$ denotes the image of a vertex $x$ under the truncation map $u$, assuming that $s \neq \sigma(x)$.
Figure~\ref{fig:2qBMG_tree} shows an example of four vertices.

\section{Basic properties of un2qBMGs and least resolved trees}
\label{sec:LRTs}

The study of \emph{undirected underlying graph of a 2qBMG (un2qBMG)} was initiated 
in~\cite{korchmaros2025forbidden}. In the latter, it was shown that
un2qBMGs do not have any induced path of size at least $6$, and they are
chordal bipartite~\cite[Corollary 3.3]{korchmaros2025forbidden},
i.e., un2qBMGs do not admit any induced cycle of length at least $6$. The
tree $(T,\sigma,u_T)$ explains the graph $G$ if $G$ is the underlying
undirected graph of 2qBMG($T,\sigma,u_T$).

\begin{definition}\label{def:un2qbmg-tree}
  Let $T$ be a phylogenetic tree with a binary leaf-coloring $\sigma$ and
  the truncation map $u$ such that $u(x)\in \{x,\rho\}$ for all $x \in
  L(T)$. A graph $G$ is \emph{explained} by $(T,\sigma, u)$ if the
  following two properties hold:
  \begin{itemize}
  \item[(i)] $V(G)=L(T)$;
  \item[(ii)] $xy\in E(G)$ whenever $\sigma(x) \neq \sigma(y)$ and
    \begin{itemize}
    \item [(ii.1)] $u(x)\neq x$ and $y \preceq \lca(x,z)$ for all $z \in
      L(T)$ with $\sigma(z)=\sigma(y)$, or
    \item [(ii.2)] $u(y)\neq y$ and $x \preceq \lca(y,z)$ for all $z \in
      L(T)$ with $\sigma(z)=\sigma(x)$.
    \end{itemize} 
  \end{itemize}
\end{definition}
Any graph $G$ satisfying Definition~\ref{def:un2qbmg-tree} is necessarily
bipartite, since condition~(ii) requires that $\sigma(x) \neq \sigma(y)$
for every edge $xy \in E(G)$.

Recall that qBMGs are directed graphs explained by a tree. In analogy,
un2qBMGs are precisely the undirected graphs explained by a tree, as stated
in the following result.
\begin{proposition}\label{prop:iff_explaing-tree}
  Let $(G,\sigma)$ be an undirected graph. $G$ is un2qBMG if and only if
  there is a tree $(T,\sigma,u)$ explaining $G$.
\end{proposition}
\begin{proof}
  Assume first that $G$ is un2qBMG. Let $\dG$ be a 2qBMG such that $G$ is
  its underlying undirected graph. Consider a tree $(T,\sigma,u)$
  explaining $\dG$. We show that $(T,\sigma,u)$ also explains $G$. By
  definition, $V(G)=V(\dG)=L(T)$, so (i) in Definition~\ref{def:un2qbmg-tree}
  holds. To see that (ii) in Definition~\ref{def:un2qbmg-tree} holds, let
  $xy\in E(G)$. Since $G$ is the underlying undirected graph of $\dG$, one
  of $xy\in E(\dG)$ or $yx\in E(\dG)$ holds. Suppose that $xy\in
  E(\dG)$. Therefore, $y\preceq \lca(x,z)$ with $\sigma(z)=\sigma(y)$ and
  $\lca(x,y)\preceq u(x)$ yielding $u(x)\neq x$. Hence, (ii.1) in
  Definition~\ref{def:un2qbmg-tree} is satisfied. If $yx\in E(\dG)$ and
  the roles of $x$ and $y$ are exchanged, then the same argument implies
  that condition (ii.2) in Definition~\ref{def:un2qbmg-tree} is satisfied.
    
  Conversely, suppose that there is a tree $(T,\sigma,u)$ explaining $G$,
  and consider the 2qBMG $\dG$ explained by $(T,\sigma,u)$. Clearly,
  $V(\dG)=L(T)=V(G)$ holds. It remains to show that for all $xy\in E(G)$,
  one of $xy \in E(\dG)$ or $yx \in E(\dG)$. So, let $xy \in E(G)$. If (ii.1) in
  Definition~\ref{def:un2qbmg-tree} is satisfied for $xy$, then
  $u(x)\neq x$ must hold, so $u(x)=\rho$ and $\lca(x,y)\preceq u(x)$
  follows. This, together with the fact that $y \preceq (x,z)$ for all $z
  \in X$ with $\sigma(z)=\sigma(y)$, implies that $xy\in E(\dG)$. One
  easily verifies that analogous arguments imply $yx \in E(\dG)$ if case
  (ii.2) in Definition~\ref{def:un2qbmg-tree} is satisfied.
  \end{proof}

\begin{lemma}\label{lm:union_un2qbmg}
  The disjoint union $(G, \sigma) = \bigcup_{i=1,\ldots m}(G_i, \sigma_i)$
  with $|\sigma_i(G_i)|=2$ is an un2qBMG if and only if each $(G_i,
  \sigma_i)$ is an un2qBMG.
\end{lemma}

\begin{proof}
Suppose $G$ is an un2qBMG. Since the property of being an un2qBMG is
hereditary~\cite{korchmaros2025forbidden}, every $G_i$ is an un2qBMG.
Conversely, assume that $G$ is not an un2qBMG, but every $G_i$ is an
un2qBMG. Proposition~\ref{prop:iff_explaing-tree} ensures that every $G_i$
is explained by a tree $(T_i,\sigma_i,u_i)$. Then, consider a star-tree $T$
with $x_1,\ldots,x_m$ leaves. Modify $T$ by rooting $T_i$ in $x_i$, and
define a truncation map $u$ on $L(T)$, as the piecewise function of
$u_1,\ldots, u_n$. Definition~\ref{def:un2qbmg-tree} togheter with the
assumption that $|\sigma_i(G_i)|=2$ for every element of the union gives
that $G$ is explained by $(T,\sigma,u)$.
\end{proof}

In this contribution, we will first prove our main results assuming that
$(G,\sigma)$ is a connected properly vertex-colored graph, and then use
Lemma~\ref{lm:union_un2qbmg} to extend the results to the case of
disconnected graphs.

For $v$ a vertex of $T$, we say that $T(v)$ is \emph{monochromatic} if all
leaves of $T(v)$ have the same color. We have:

\begin{proposition}\label{prop:lrt_lca}
  Let $G$ be a connected un2qBMG explained by the tree $(T,\sigma,u)$. If
  $(T,\sigma,u)$ is least-resolved, then for every internal vertex $v$ of
  $T$, there exists an edge $xy \in E(G)$ such that $v=\lca_T(x,y)$.
\end{proposition}

\begin{proof}
 By contradiction, suppose that there exists an internal vertex $v$ of $T$
  such that no edge $xy \in E(G)$ satisfies $v=\lca_T(x,y)$. We distinguish
  between two cases, according as $v=\rho$ and $v \neq \rho$.

  Suppose first that $v=\rho$, and let $v_1, \ldots, v_k$, $k \geq 2$ be
  the children of $v$. Observe that for every $x\in L(v_i)$ and $y\in
  L(v_j)$ with $i\neq j$, $\lca(x,y)=\rho$. From our assumption on $v$,
  there is no path between $x$ and $y$ in $G$ when
  $\sigma(x)\neq\sigma(y)$. This contradicts $G$ being connected.

  Suppose now that $v\neq \rho$. Denote by $w$ be the parent of $v$, and by
  $T'$ the tree obtained from $T$ by contracting the arc $wv$. We show that
  $G=G'$, where $G'$ is explained by $(T',\sigma,u)$, contradicting the
  assumption that $(T,\sigma,u)$ is least-resolved. Observe that
  $V(G)=V(G')$ by Definition~\ref{def:un2qbmg-tree}(i). Hence, we proceed
  to show that $E(G)=E(G')$. For every $xy\in E(G)$, we may assume that
  $u(x) \neq x$, and $y \preceq_T \lca_T(x,z)$ for all $z \in V(G)$ with
  $\sigma(z)=\sigma(y)$ from Definition~\ref{def:un2qbmg-tree}. Now, let
  $z$ be such that $\sigma(z)=\sigma(y)$. Since $y \preceq_T \lca_T(x,z)$,
  contracting $wv$ yields $y \preceq_{T'} \lca_{T'}(x,z)$. Thus, $xy\in
  E(G')$.

  Conversely, suppose that $xy \notin E(G)$. We may assume that $\sigma(x)
  \neq \sigma(y)$, and $u(x) \neq x$ or $u(y) \neq y$. Indeed, in both
  cases we have $xy\notin E(G')$. Without loss of generality, assume that
  $u(x) \neq x$. $G$ being connected together with $xy \notin E(G)$ yields
  that there exists $z\in V(G)$ with $\sigma(z)=\sigma(y),\, xz\in E(G)$
  and $y$ is not a descendant of $\lca_T(x,z)$. Thus, $xy\in E(G')$ if and
  only if $\lca_T(x,z)=v$ and $y$ is a descendant of $w$ but not of $v$
  since $T'$ is obtained by contraction $wv$. This contradicts the
  assumption on $v$. Hence $xy \notin E(G')$.
\end{proof}

As a direct consequence, we obtain the following result.
\begin{corollary}\label{cor:lrt-monochr}
    For every least-resolved tree $(T,\sigma, u)$, and every internal
    vertex $v$ of $T$, the subtree $T(v)$ rooted at $v$ is not
    monochromatic.
\end{corollary}
The converse of Corollary~\ref{cor:lrt-monochr} does not hold in
general. As illustrated in Figure~\ref{fig:lrt_mon}, the non-monochromatic
trees $T$ and $T'$ explain the same graph $G$. Because $T'$ results from
contracting an edge in $T$, the latter is not least-resolved.

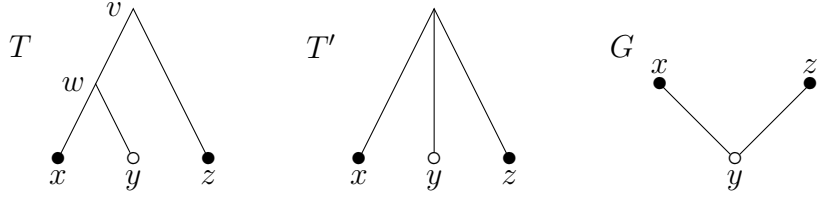
\begin{figure}
\begin{center}
\begin{tikzpicture}
\draw (0,0) node[below]{$x$} node{$\bullet$} -- (0.5,1) node[left]{$w$} -- (1,0) node[below]{$y$} node{\color{white} $\bullet$} node{$\circ$};
\draw (0.5,1) -- (1,2) node[left]{$v$} -- (2,0) node[below]{$z$} node{$\bullet$};
\draw (-0.5,1.5) node{$T$};

\draw (4,0) node[below]{$x$} node{$\bullet$} -- (5,2) -- (6,0) node[below]{$z$} node{$\bullet$};
\draw (5,2) -- (5,0) node[below]{$y$} node{\color{white} $\bullet$} node{$\circ$};
\draw (3.5,1.5) node{$T'$};

\draw (8,1) node{$\bullet$} node[above]{$x$} -- (9,0) node{\color{white} $\bullet$} node{$\circ$} node[below]{$y$} -- (10,1) node{$\bullet$} node[above]{$z$};
\draw (7.5,1.5) node{$G$};
\end{tikzpicture}
\caption{Example of non-monochromatic least-resolved tree.  
$(T,\sigma,u)$ and $(T,\sigma,u')$ explain $G$, where $u$ and $u'$ map every leaf to the root of $T$ and $T'$, respectively. Since $T'$ can be obtained from $T$ by contracting the arc $vw$, $T$ is not least-resolved, although it is not monochromatic.}
\label{fig:lrt_mon}
\end{center}
\end{figure}

We are now in a position to describe the set of neighbours of any vertex of
an un2qBMG.
\begin{proposition}\label{pr:neighbors}
  Let $G$ be an un2qBMG explained by a least-resolved tree $(T,\sigma,
  u)$. For every $xy\in E(G)$, one of the following cases holds:
  \begin{itemize}
  \item[(i)] the parent of $x$ is an ancestor of $y$ and $u(x) \neq x$;
  \item[(ii)] the parent of $y$ is an ancestor of $x$ and $u(y) \neq y$.
  \end{itemize}
\end{proposition}
\begin{proof}
  Let $xy\in E(G)$. Then either (ii.1) or (ii.2) in Definition~\ref{def:un2qbmg-tree} is satisfied. Suppose first that
  Definition~\ref{def:un2qbmg-tree}(ii.1) holds. Let $v$ be the parent of
  $x$ in $T$. We proceed to show that $v$ is an ancestor of $y$. Since
  $T(v)$ is non-monochromatic by Corollary~\ref{cor:lrt-monochr}, there
  exists $z \in L(T(v))$ such that $\sigma(z) \neq \sigma(x)$. In
  particular, we have $v=\lca(x,z)$, so the relation $y \preceq \lca(x,z)$
  implies that $y$ is a descendant of $v$. Hence, case (i) holds. Using
  similar arguments, Definition~\ref{def:un2qbmg-tree}(ii.2) yields case
  (ii).
\end{proof}

\section{Small un2qBMGs and bicliques}\label{sec:5vertices}
This section is devoted to the study of connected graphs with at most $5$
vertices, as well as complete bipartite graphs. In both cases, the graphs
turn out to be un2qBMGs.

\begin{proposition}\label{prop:bicliques}
  Complete bipartite graphs are precisely the un2qBMGs explained by
  $(T,\sigma,u)$ where $T$ is a star-tree and $u(x)= \rho$ for all $x\in
  L(T)$.
\end{proposition}
\begin{proof}
Let $(G,\sigma)$ be a biclique, and let $(T,\sigma,u)$ be a star-tree with
leaf-coloring $\sigma$, $L(T)=V(G)$, and $u(x)=\rho$ for all $x \in
L(T)$. Since $(G,\sigma)$ is a biclique, we have that $xy \in E(G)$ if and
only if $\sigma(x) \neq \sigma(y)$. Moreover, by definition of
$(T,\sigma,u)$, all pairs $x,y \in X$ with $\sigma(x) \neq \sigma(y)$
satisfy condition (ii.1) of Definition~\ref{def:un2qbmg-tree}. Hence,
$(G,\sigma)$ is explained by $(T,\sigma,u)$.\\ Conversely, suppose that
$(G,\sigma)$ is explained by $(T,\sigma,u)$. If $\sigma(x)=\sigma(y)$, then
by Definition~\ref{def:un2qbmg-tree}(ii), $xy \notin E(G)$. Otherwise,
i.e., if $\sigma(x) \neq \sigma(y)$, then $x$ and $y$ satisfy
Definition~\ref{def:un2qbmg-tree}(ii.1). Thus, $G$ is a complete bipartite
graph.
\end{proof}

\begin{proposition}\label{prop:small_cases}
  Let $(G,\sigma)$ be a connected bipartite graph and $1<|V(G)|<6$. $G$ is
  un2qBMG explained by the tree $(T,\sigma,u)$ whenever one of the
  following cases is satisfied.
\begin{itemize}
\item[(I)] $|V(G)| \in \{2,3\}$ or $|V(G)| \in\{4,5\}$ and $G$ is complete
  bipartite, $T$ is a star-tree with $L(T)=V(G)$ and $u(x):=\rho_T$ for all
  $x\in V(G)$.
\item [(II)] $|V(G)|=4$ and $G=x_1x_2x_3x_4$, $T$ is the tree in
  Figure~\ref{fig:small_trees}(a) and $u(x_1):=x_1, u(x_i):=\rho_T$ for
  $2\leq i\leq 4$;
\item [(III)] $|V(G)|=5$ and
  \begin{enumerate}
  \item[(1)] $E(G)=\{x_1x_2,x_2x_3,x_3x_4, x_4x_5\}$, $T$ is the tree in
    Figure~\ref{fig:small_trees}(b) and $u(x_i):=\rho_T$ for $1\leq i \leq
    5$;
  \item[(2)] $E(G)=\{x_1x_4,x_2x_3,x_3x_4,x_4x_5\}$, $T$ is the tree in
    Figure~\ref{fig:small_trees}(c), and $u(x_1)=x_1, u(x_5)=x_5,
    u(x_i):=\rho_T$ for $2 \leq i\leq 4$;
  \item[(3)] $E(G)=\{x_1x_2,x_1x_4,x_2x_3,x_3x_4,x_4x_5\}$, $T$ is the tree
    in Figure~\ref{fig:small_trees}(d) (top), and $u(x_5):=x_5,
    u(x_i):=\rho_T$ for $1\leq i\leq 4$.
  \end{enumerate}
 \end{itemize}
 \end{proposition}

 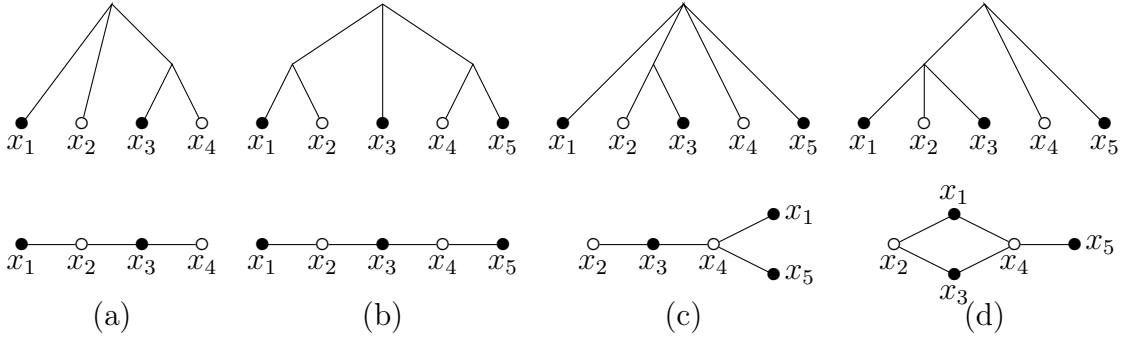
\begin{figure}
     \centering
     \begin{tikzpicture}[scale=0.8]
     \draw (0,-1.2) node{(a)};
     \draw (-1.5,0) -- (1.5,0);
     \draw (-1.5,0) node{$\bullet$} node[below]{$x_1$};
     \draw (-0.5,0) node{\color{white} $\bullet$} node{$\circ$} node[below]{$x_2$};
     \draw (0.5,0) node{$\bullet$} node[below]{$x_3$};
     \draw (1.5,0) node{\color{white} $\bullet$} node{$\circ$} node[below]{$x_4$};
     \draw (-1.5,2) node[below]{$x_1$} node{$\bullet$} -- (0,4) -- (-0.5,2) node[below] {$x_2$} node{\color{white} $\bullet$} node{$\circ$};
     \draw (0.5,2) node[below]{$x_3$} node{$\bullet$} -- (1,3) -- (1.5,2) node[below]{$x_4$} node{\color{white} $\bullet$} node{$\circ$};
     \draw (0,4) -- (1,3);

     \begin{scope}[xshift=4.5cm]
     \draw (0,-1.2) node{(b)};
     \draw (-2,0) -- (2,0);
     \draw (-2,0) node{$\bullet$} node[below]{$x_1$};
     \draw (-1,0) node{\color{white} $\bullet$} node{$\circ$} node[below]{$x_2$};
     \draw (0,0) node{$\bullet$} node[below]{$x_3$};
     \draw (1,0) node{\color{white} $\bullet$} node{$\circ$} node[below]{$x_4$};
     \draw (2,0) node{$\bullet$} node[below]{$x_5$};
     \draw (-2,2) node[below]{$x_1$} node{$\bullet$} -- (-1.5,3) -- (-1,2) node[below]{$x_2$} node{\color{white} $\bullet$} node{$\circ$};
     \draw (0,4) -- (0,2) node[below]{$x_3$} node{$\bullet$};
     \draw (1,2) node[below]{$x_4$} node{\color{white} $\bullet$} node{$\circ$} -- (1.5,3) -- (2,2) node[below]{$x_5$} node{$\bullet$};
     \draw (-1.5,3) -- (0,4) -- (1.5,3);
     \end{scope}

     \begin{scope}[xshift=9.5cm]
     \draw (0,-1.2) node{(c)};
     \draw (-1.5,0) -- (0.5,0);
     \draw (1.5,0.5) -- (0.5,0) -- (1.5,-0.5);
     \draw (-1.5,0) node{\color{white} $\bullet$} node{$\circ$} node[below]{$x_2$};
     \draw (-0.5,0) node{$\bullet$} node[below]{$x_3$};
     \draw (0.5,0) node{\color{white} $\bullet$} node{$\circ$} node[below]{$x_4$};
     \draw (1.5,0.5) node{$\bullet$} node[right]{$x_1$};
     \draw (1.5,-0.5) node{$\bullet$} node[right]{$x_5$};
     \draw (-2,2) node[below]{$x_1$} node{$\bullet$} -- (0,4) -- (-0.5,3);
     \draw (-1,2) node[below]{$x_2$} node{\color{white} $\bullet$} node{$\circ$} -- (-0.5,3) -- (0,2) node[below]{$x_3$} node{$\bullet$};
     \draw (1,2) node[below]{$x_4$} node{\color{white} $\bullet$} node{$\circ$} -- (0,4) -- (2,2) node[below]{$x_5$} node{$\bullet$};
     \end{scope}

     \begin{scope}[xshift=14.5cm]
     \draw (0,-1.2) node{(d)};
     \draw (1.5,0) -- (0.5,0);
     \draw (0.5,0) -- (-0.5,0.5) -- (-1.5,0) -- (-0.5,-0.5) -- cycle;
     \draw (-1.5,0) node{\color{white} $\bullet$} node{$\circ$} node[below]{$x_2$};
     \draw (-0.5,-0.5) node{$\bullet$} node[below]{$x_3$};
     \draw (-0.5,0.5) node{$\bullet$} node[above]{$x_1$};
     \draw (0.5,0) node{\color{white} $\bullet$} node{$\circ$} node[below]{$x_4$};
     \draw (1.5,0) node{$\bullet$} node[right]{$x_5$};
     \draw (-2,2) node[below]{$x_1$} node{$\bullet$} -- (-1,3) -- (0,4);
     \draw (-1,2) node[below]{$x_2$}  node{\color{white} $\bullet$} node{$\circ$}-- (-1,3) -- (0,2) node[below]{$x_3$} node{$\bullet$};
     \draw (1,2) node[below]{$x_4$} node{\color{white} $\bullet$} node{$\circ$} -- (0,4) -- (2,2) node[below]{$x_5$} node{$\bullet$};
     \end{scope}
     \end{tikzpicture}
     \caption{All connected non-biclique un2qBMGs with at most five
       vertices, together with the trees explaining them. Bipartitions
       $\sigma$ are shown as vertex-colorings. The truncation maps $u$
       are: (a) $u(x_1)=x_1$ and $u(x_i)=\rho_T$ for $i \in \{2,3,4\}$; (b)
       $u(x_i)=\rho_T$ for $i \in \{1, \ldots, 5\}$; (c) $u(x_1)=x_1$,
       $u(x_5)=x_5$ and $u(x_i)=\rho_T$ for $i \in \{2,3,4\}$; (d)
       $u(x_5)=x_5$ and $u(x_i)=\rho_T$ for $i \in \{1, \ldots, 4\}$.}
     \label{fig:small_trees}
 \end{figure}

\begin{proof}
  \noindent \begin{itemize}
  \item [(I)] $|V(G)| \in \{4,5\}$ and $G$ is a biclique, or $|V(G)| \in
    \{2,3\}$. Observe that $G$ is also a biclique when $|V(G)| \in \{2,3\}$
    since $G$ is connected. Proposition~\ref{prop:bicliques} ensures that
    $(G,\sigma)$ is exactly the un2qBMG explained by the tree
    $(T,\sigma,u)$ where $T$ is a star-tree with root $\rho_T$ and leaf set
    $L(T)=V(G)$, and $u(x):=\rho_T$ for all $x\in V(G)$.
  \item [(II)] $V(G)=\{x_1,x_2,x_3,x_4\}$ and
    $E(G)=\{x_1x_2,x_2x_3,x_3x_4\}$; see
    Figure~\ref{fig:small_trees}(a). Observe that setting $T$ as the tree
    in Figure~\ref{fig:small_trees}(a) and $u(x_1):=x_1, u(x_i):=\rho_T$
    for $2\leq i\leq 4$ gives a $(T,\sigma,u)$ explaining $G$ from
    Definition~\ref{def:un2qbmg-tree}. Hence, $G$ is an un2qBMG.
  \item[(III)] $V(G)=\{x_1,x_2,x_3,x_4,x_5\}$. In each of the three
    subcases, it is straightforward to check that $(T,\sigma,u)$ explains
    $G$ by Definition~\ref{def:un2qbmg-tree}. Hence $G$ is an un2qBMG.
\end{itemize}
\end{proof}

It is straightforward to verify that if $G$ is a bipartite, connected graph
with at most five vertices, its vertex-set and edge-set fall into one of
the cases (I), (II), or (III) of Proposition~\ref{prop:small_cases}. As a
consequence, the following result is obtained.

\begin{corollary}\label{cor:forbidden5}
  Every connected bipartite graph with at most $5$ vertices is an un2qBMG.
\end{corollary}
An immediate consequence of Corollary~\ref{cor:forbidden5} is that there
are no forbidden induced subgraphs of a un2qBMG with at most $5$ vertices.
For $|V|=6$, both $P_6$ and $C_6$ are bipartite but not un2qBMGs.  In
Section~\ref{sec:forbidden} we will provide a characterization of un2qBMGs
in terms of forbidden induced subgraphs.
   
\section{Recognition of un2qBMGs}\label{sec:algorithm}

In this section, we present a polynomial-time algorithm for the recognition
of un2qBMGs. Consider a bipartite graph $(G,\sigma)$ with a vertex-coloring
$\sigma$ induced by the bipartition of the vertex-set $V(G)$. The algorithm
either builds a tree $(T,\sigma,\tau)$ explaining $G$, or returns the
statement that $G$ is not a un2qBMG. The following definition plays a key
role in the algorithm.

\begin{definition}
Let $(G,\sigma)$ be a bipartite graph. A vertex $x$ of $G$ that is adjacent
to all vertices of the opposite color of $G$ is called a
\emph{heart-vertex}.
\end{definition}

We begin with the following result for connected un2qBMGs.
\begin{lemma}\label{lm:vheart}
Every connected un2qBMG has a heart-vertex.
\end{lemma}

\begin{proof}
  Let $(G,\sigma)$ be a connected un2qBMG explained by a least-resolved
  tree $(T,\sigma,u)$. Proposition~\ref{prop:lrt_lca} implies that
  there exist $x,y \in V(G)$ such that $xy \in E(G)$ and
  $\rho_T=\lca_T(x,y)$.  By Proposition~\ref{pr:neighbors}, we may assume
  that $u(x) \neq x$ and that the parent $v$ of $x$ is an ancestor of
  $y$. Hence, $v = \lca(x,y) = \rho_T$, which implies that $x$ is a child
  of $\rho_T$. Together with $u(x) \neq x$, this yields $xz \in E(G)$ for
  every $z \in L(T)$ with $\sigma(z) \neq \sigma(x)$; that is, $x$ is a
  heart-vertex of $G$.
\end{proof}

Note that a un2qBMG $G$ needs not be connected. For example, the un2qBMG
explained by the tree $(T,\sigma,u)$ depicted in
Figure~\ref{fig:2qBMG_tree} has two connected components. However, all
connected components of $G$ are un2qBMG by Lemma~\ref{lm:union_un2qbmg},
and therefore, Lemma~\ref{lm:vheart} implies that all connected components
of $G$ have a heart-vertex. Lemma~\ref{lm:vheart} lies at the heart of our
algorithm \textsc{HEART-TREE}, which we present now.

\begin{algorithm}[ht!]
\caption{\textsc{HEART-TREE}} 
\label{alg:main} 
\begin{algorithmic}[1] 
\Require{A bipartite graph $G$.}
\Ensure{A tree $(T,\sigma,u)$ explaining $G$ if $G$ is a un2qBMG. Otherwise, \texttt{false}.}
\State{Let $\sigma$ be the vertex-coloring map induced by the bipartition of $V(G)$}\label{l:inits}
\State{$T:=\rho,\,F(\rho):=V(G)$} \label{l:initL}  
\If {$|V(G)|=1$}
\State{$x:=\rho$ and $u(x):=x$ with $V(G)=\{x\}$} %Set the unique $x \in V(G)$ equal to $\rho$ and $u(x):=x$
\Else
\While{$L(T)\setminus V(G)\neq \emptyset$}\label{l:while}
\State{Take $v\in L(T)\setminus V(G)$}
\State{$G_v:=G[F(v)]$ and $H_v:=\{x \in F(v): x \mbox{ is a heart-vertex of } G_v\}$}\label{l:defGHv} 
\If{$H_v=\emptyset$ and $G_v$ is connected} ~\label{l:ifnope}
\Return{\texttt{false}}
\Else
\For{$x \in H_v$}
\State{Add $x$ as a child of $v$ and $u(x):=\rho$}\label{l:lcr}
\EndFor
\State{$G_v^-:=G_v[F(v) \setminus H_v]$}\label{l:defGv-}
\For{every connected components $C$ of $G_v^-$}\label{l:conGv-}
\If{$|V(C)|=1$}
\State{Add the unique $x \in V(C)$ as a child to $v$ and  $u(x):=x$}\label{l:lcx}
\Else
\State{Add a new child $w$ to $v$ and $F(w):=V(C)$} \label{l:tc}
\EndIf
\EndFor
\EndIf
\EndWhile
\EndIf
\Return $(T,\sigma,u)$
\end{algorithmic}
\end{algorithm}

\begin{theorem}\label{th:algoworks}
  Let $(G,\sigma)$ be a bipartite graph. Algorithm \textsc{HEART-TREE}
  returns a tree $(T,\sigma,u)$ if and only if $G$ is a un2qBMG. Moreover,
  in that case, the tree $(T,\sigma,u)$ returned by \textsc{HEART-TREE}
  explains $G$ and is least-resolved.
\end{theorem}

\begin{proof}
Suppose first that Algorithm~\ref{alg:main} returns \texttt{false}. Then,
there exists an instance of the loop initiated at Line~\ref{l:while} in
which $H_v=\emptyset$ and $G_v$ is connected (Line~\ref{l:ifnope}), where
$H_v$ and $G_v$ are defined at Line~\ref{l:defGHv}. From
Lemma~\ref{lm:vheart}, $G_v$ is not a un2qBMG. Hence, $G$ is not an
un2qBMG, as $G_v$ is an induced subgraph of $G$, and the property of being
an un2qBMG is hereditary~\cite[Proposition
  3.1]{korchmaros2025forbidden}.

Conversely, suppose that Algorithm~\ref{alg:main} returns a tree
$(T,\sigma,u)$. Showing that $G$ is a un2qBMG is equivalent to verifying
that $(T,\sigma,u)$ explains $G$ by
Proposition~\ref{prop:iff_explaing-tree}. Let $G_T$ be the un2qBMG
explained by $(T,\sigma,u)$. We proceed to show that $G=G_T$ by verifying
that (I) $V(G)=V(G_T)$ and (II) $E(G)=E(G_T)$.
\begin{itemize}
\item[(I)] $V(G)=V(G_T)$. $V(G_T)=L(T)$ follows from
  Definition~\ref{def:un2qbmg-tree}(i); therefore, it remains to show that
  $V(G)=L(T)$. From the condition on Line~\ref{l:while}, $L(T) \subseteq
  V(G)$ must hold to exit the loop. We now show that $V(G) \subseteq L(T)$
  also holds. First, we claim that for all internal vertices $v$ such that
  $x \in F(v)$, either $x$ is a child of $v$, or there exists exactly one
  child $w$ of $v$ such that $x \in F(w)$. To see the claim, let $G_v$,
  $H_v$ and $G_v^-$ be as defined on Lines~\ref{l:defGHv} and
  \ref{l:defGv-}. If $x \in H_v$, then $x$ is a child of $v$
  (Line~\ref{l:lcr}), and the claim is verified. Otherwise, $x$ is a vertex
  of $G_v^-$, and there exists a (necessarily unique) connected component
  $C$ of $G_v^-$ containing $x$ in its vertex-set. If $V(C)=\{x\}$, then
  $x$ is a child of $v$ (Line~\ref{l:lcx}), and the claim is
  verified. Otherwise, $v$ has a child $w$ that satisfies $F(w)=V(C)$, and
  the claim is satisfied since $x \in V(C)$. The claim being true, and
  since $x \in F(\rho)$ (Line~\ref{l:initL}), it follows that there exists
  a (necessarily unique) leaf $v$ of $T$ such that either $u=x$ or $x \in
  F(u)$ holds. In view of the condition at Line~\ref{l:while}, the latter
  is impossible, so we have $x \in L(T)$, and $V(G) \subseteq L(T)$ holds
  as desired. This concludes the proof of the claim that $V(G)=V(G_T)$ holds.

\item[(II)] $E(G)=E(G_T)$. Let $xy\in E(G_T)$ with
  $v:=\lca(x,y)$. Line~\ref{l:tc} ensures $x,y \in F(v)\subseteq
  V(G_v)$. Note that $xy \in E(G_T)$ implies $\sigma(x)\neq \sigma (y)$. We
  show that at least one of $x$ or $y$ is an element of $H_v$. Assume that
  $x,y\notin H_v$. Hence, $x,y\in V(G_v^-)$ (Line~\ref{l:defGv-}). Let
  $C_x$ and $C_y$ be the connected components of $G_v^-$ that contain $x$
  and $y$, respectively. If $C_x=C_y$, then there exists a child $w$ of $v$
  such that $x,y \in F(w)$ (Line~\ref{l:tc}). In particular, $w$ is an
  ancestor of both $x$ and $y$ in $T$, contradicting $v=\lca(x,y)$. Hence,
  $C_x \neq C_y$. Up to a permutation, we may assume that $u(x)\neq x$ from
  Definition~\ref{def:un2qbmg-tree}(ii), as $xy\in E(G_T)$. Observe that
  $|C_x|>2$. Indeed, if $|V(C_x)|=1$, then $x$ is the unique element of
  that set, and $u(x)=x$ (Lines~\ref{l:lcx}), contradicting the assumption
  that $u(x)\neq x$. Connectedness of $C_x$ implies that there exists $z
  \in V(C_x)$ such that $xz \in E(G)$ and $\sigma(z) \neq \sigma(x)$
  (Line~\ref{l:inits}). Let $w$ be the child of $v$ such that $L(T(w))=C_x$
  (Line~\ref{l:tc}). Thus, $\lca(x,z)\preceq w \prec v=\lca(x,y)$,
  contradicting Definition~\ref{def:un2qbmg-tree}(ii.1). Hence, at least
  one of $x$ or $y$ must be an element of $H_v$. From the definition of
  $H_v$ (Line~\ref{l:defGHv}), $xy \in E(G_v)$. In addition, $xy \in E(G)$
  since $G_v$ is an induced subgraph of $G$ (Line~\ref{l:defGHv}). Hence
  $E(G_T) \subseteq E(G)$.

  To show that $E(G) \subseteq E(G_T)$, take $xy\in E(G)$ with
  $v:=\lca(x,y)$. Line~\ref{l:inits} ensures that $\sigma(x) \neq
  \sigma(y)$. From Lines~\ref{l:lcr},~\ref{l:lcx}, and \ref{l:tc}, we have
  $x,y \in F(v)$; thus $x$ and $y$ are vertices of $G_v$. Suppose $x,y\in
  V(G_v^-)$. Hence $xy\in E(G_v^-)$ since $xy\in E(G)$ and $G_v^-$ is an
  induced subgraph of $G$ (Line~\ref{l:defGHv}). Therefore, $x$ and $y$
  belong to the same connected component $C$ of $G_v^-$. In particular,
  there exists a child $w$ of $v$ such that $x,y \in F(w)$
  (Line~\ref{l:tc}). This implies that $x$ and $y$ are both descendant of
  $w$, contradicting $v=\lca(x,y)$. Hence, at least one between $x$ and $y$
  is an element of $H_v$. Without loss of generality, we may assume that $x
  \in H_v$. Then, $x$ is a child of $v$, and we have $u(x) \neq x$
  (Line~\ref{l:lcr}). Take $z\in L(T)$ with $\sigma(z)=\sigma(y)$. $z\notin
  V(G_v)$ yields $v\prec\lca(x,z)$ (Line~\ref{l:tc}). On the other hand,
  $\lca(x,z)=v$ whenever $z\in V(G_v)$. Therefore, in both cases, $y\prec
  \lca(x,y)=v\preceq\lca(x,z)$. Hence, $xy\in E(G_T)$ from
  Definition~\ref{def:un2qbmg-tree}(ii.1).
\end{itemize}
It remains to show that $(T,\sigma,u)$ is least resolved. Let $vw$ be an
internal arc of $T$, and let $T'$ be the tree obtained from $T$ by
contracting the arc $vw$. Let $G'$ be the un2qBMG explained by
$(T',\sigma,u)$. We show that $G\neq G'$.  Observe that $F(w)$ induces a
connected subgraph of $G$ of size $2$ or more, as $v$ is not a leaf of
$T$. In particular, $G_w$ has a heart-vertex $x$ by
Lemma~\ref{lm:vheart}. Moreover, there exists a vertex $y \in F(v)$ such
that $\sigma(y) \neq \sigma(x)$ and $xy \notin E(G)$. Indeed, if no such
vertex exists, then $x$ is a heart-vertex of $G_v$, contradicting the
choice of $x$ as an element of $F(w) \subseteq F(v) \setminus H_v$. Define
$\lca_{T'}(x,y)=\overline v$, where $\overline v$ is the vertex resulting
from the contraction of the arc $vw$. Observe that $u(x)\neq x$ since $x$
is a heart-vertex of $G_w$. Thus, $xy\in E(G')$ from
Definition~\ref{def:un2qbmg-tree}~(ii.1) but $xy\notin E(G)$. Hence, $G\neq
G'$.
\end{proof}

\begin{proposition}\label{prop:complexity_alg}
  \textsc{HEART-TREE} algorithm runs in $O(n^3)$ where $n=V(G)$.
\end{proposition}
\begin{proof}
  In each iteration of the while-loop the algorithm processes a leaf $v$ by
  creating the induced subgraph $G_v:=G[F(v)]$ (Line~\ref{l:defGHv}),
  finding the set $H_v$ of heart-vertices of $G_v$ (Line~\ref{l:defGHv}),
  creating $G_v^-:=G_v[F(v)\setminus H_v]$ (Line~\ref{l:defGv-}), and
  computing its connected components (Line~\ref{l:conGv-}). These steps can
  be achieved at once by scanning the adjacency lists of the vertices in
  $F(v)$ and traversing $G_v$ with Breadth-First Search~\cite{CLRS2009},
  hence in $O(|F(v)|+|E(G_v)|)$.  Let $\mathcal{P}$ denote the (finite) set
  of elements $v$ that are processed by the loop. The running time is
  \begin{equation}\label{eq1}
    \sum_{v\in\mathcal P} O(|F(v)|+|E(G_v)|\big)
    \;\leq\; O\!\left(\sum_{v\in\mathcal P} |F(v)|\right)
    + O\!\left(\sum_{v\in\mathcal P} |E(G_v)|\right).
  \end{equation}
  Recall that the set $\mathcal{P}$ of elements $v$ processed by the loop
  is precisely the set $V(T) \setminus L(T)$, where $T$ is the tree
  returned by \textsc{HEART-TREE}. On the other hand, $T$ is a tree, so we
  have $|V(T)| \leq 2|L(T)|-2=O(|L(T)|)=O(n)$.
  This, together with $F(v)\subseteq V(G)$ (Line~\ref{l:tc}) and
  $E(G_v)\subseteq E(G)$ (Line~\ref{l:defGHv}) yields
  \begin{equation}\label{eq2}
    \begin{split}
    \sum_{v\in\mathcal P} |F(v)| &\leq \sum_{v\in V(T)} |F(v)| \leq O(n^2)
    \\
    \sum_{v\in\mathcal P} |E(G_v)|&\leq\sum_{v\in V(T)} |E(G_v)|\leq O(n^3).
    \end{split}
    \end{equation}
  Combining Equations \eqref{eq1} and \eqref{eq2} yields a total running
  time of $O(n^3)$.
\end{proof}

An immediate consequence of Proposition~\ref{prop:complexity_alg} together
with Theorem~\ref{th:algoworks} is the following result on the
computational complexity of un2qBMGs.
\begin{corollary}
  For a bipartite graph $G$, it is decidable in cubic time whether $G$ is
  an un2qBMG.
\end{corollary}

Equipped with \textsc{HEART-TREE} algorithm, we arrive at the following characterization of un2qBMGs.
\begin{theorem}\label{th:charh}
  Let $G$ be a bipartite graph. Then $G$ is a un2qBMG if and only if all
  connected induced subgraphs of $G$ contain a heart-vertex.
\end{theorem}
\begin{proof}
  Suppose $G$ is a un2qBMG explained by an LRT $(T,\sigma,u)$. Take a
  connected induced subgraph $H$ of $G$. $H$ is also an un2qBMG, as being
  an un2qBMG is hereditary~\cite[Proposition 3.1]{korchmaros2025forbidden}. Hence $H$ has a
  heart-vertex from Lemma~\ref{lm:vheart}.  Conversely, suppose that $G$ is
  not an un2qBMG. By Theorem~\ref{th:algoworks}, \textsc{HEART-TREE}
  returns \texttt{false} when applied to $G$. Hence, there exists an
  instance of the loop initiated at Line~\ref{l:while} in which
  $H_v=\emptyset$ and $G_v$ is connected (Line~\ref{l:ifnope}, where $H_v$
  and $G_v$ are defined at Line~\ref{l:defGHv}). In particular, $G_v$ is a
  connected induced subgraph of $G$ that does not contain a heart-vertex.
\end{proof}

\section{Forbidden subgraphs characterization}\label{sec:forbidden}
In this section, we show that un2qBMGs are not properly contained in the
family of $(P_6,C_6)$-free graphs, as they are characterized as those
graphs free of induced $P_6,C_6$, and Sunlet$_4$ graphs.

\begin{lemma}\label{lm:maxpr}
  Let $(G,\sigma)$ be a connected bipartite graph. Let $x \in V(G)$ such
  that $\deg_G(x)$ is maximal among all vertices. If $G$ is $P_6$-free,
  then for every $x' \in V(G)$ with $\sigma(x)=\sigma(x')$, there exists a
  vertex $y$ with $\sigma(y)\neq \sigma(x)$ such that $xy,x'y\in E(G)$.
\end{lemma}

\begin{proof}
  Let $x'\in V(G)$ with $\sigma(x')=\sigma(x)$. Since $G$ is connected,
  there exists a path $P$ between $x$ and $x'$ in $G$. Without loss of
  generality, we may assume that $P$ is a shortest path between $x$ and
  $x'$. Since $G$ is bipartite and $\sigma(x')=\sigma(x)$, $P$ contains an
  odd number of vertices. Moreover, $G$ is $P_6$-free, so
  $|V(P)|<6$. Hence, we have $|V(P)| \in \{3,5\}$, so either $P=xyx'$ for
  some vertex $y$ with $\sigma(y)\neq \sigma(x)$, or $P=xy\tilde x\tilde
  yx'$ for some vertex $\tilde x$ with $\sigma(\tilde x)=\sigma(x)$ and
  vertices $y,\tilde y$ with $\sigma(\tilde y)=\sigma(y)\neq\sigma(x)$.

  Suppose that the latter holds. In particular, $\tilde y$ is not adjacent
  to $x$. This, together with $\deg_G(x) \geq \deg_G(\tilde x)$ implies the
  existence of a vertex $y'$ such that $y'x\in E(G)$ but $y'\tilde x\notin
  E(G)$. On the other hand, $y'x'\notin E(G)$, because otherwise $xy'x'$ is
  a path between $x$ and $x'$ in $G$, contradicting the choice of $P$ as a
  shortest path between $x$ and $x'$. However, then $y'x\tilde y\tilde x
  yx'$ is an induced $P_6$ of $G$, contradicting the assumption that $G$ is
  $P_6$-free. Hence, this case cannot occur.

  Therefore, $P$ must be of the form $xyx'$. In particular, $y$ is adjacent
  to both $x$ and $x'$ in $G$, which concludes the proof.
\end{proof}

\begin{proposition}\label{pr:heartless}
  Let $(G,\sigma)$ be a connected bipartite graph. If $G$ contains no
  $P_6$, $C_6$, or $\text{Sunlet}_4$ as an induced subgraph, then $G$ has a
  heart-vertex.
\end{proposition}

\begin{proof}
  Assume that $G$ does not contain a heart-vertex.

  Let $x_0 \in V(G)$ be such that $\deg_G(x_0)$ is maximal among all
  elements in $V(G)$. Since $G$ has no heart-vertex, there exists a vertex
  $y_0$ with $\sigma(x_0)\neq \sigma(y_0)$ such that $x_0y_0\notin
  E(G)$. Let $P$ be a shortest path between $x_0$ and $y_0$ in $G$. Note
  that since $G$ is bipartite, $|V(P)|$ is even. Moreover, $G$ does not
  contain an induced $P_6$, so $|V(P)|<6$ must hold. Since $|V(P)|\neq 2$
  by the choice of $y_0$, this implies that only $|V(P)|=4$ is
  possible. Hence, $P=x_0y_1x_1y_0$ for some vertices $x_1,y_1$ with
  $\sigma(x_1)=\sigma(x_0)$ and $\sigma(y_1)=\sigma(y_0)$.

  Recall that, by the choice of $x_0$, $\deg(x_0) \geq \deg(x_1)$. On the
  other hand, $x_1y_1,x_0y_1,x_1y_0\in E(G)$ but $x_0y_0\notin E(G)$. This
  implies that there exists $y_2 \in V(G)$ such that $y_2$ is adjacent to
  $x_0$ and not to $x_1$. In particular, $y_2$ is distinct from $y_0$ and
  $y_1$, and we have $\sigma(y_2)=\sigma(y_0)$. Hence, $y_2x_0y_1x_1y_0$ is
  an induced $P_5$ in $G$, see Figure~\ref{fig:fpr}~(a).

  Since $G$ contains no heart-vertex, there is a vertex $x_2$ with
  $\sigma(x_2)=\sigma(x_0)$ such that $x_2y_1\notin E(G)$. In particular,
  $x_2$ is distinct from $x_0$ and $x_1$. From the maximality of the degree
  of $x_0$, Lemma~\ref{lm:maxpr} applies to $x_0$ and $x_2$, and there
  exists $y_3$ with $\sigma(y_3)\neq \sigma(x_0)$ such that
  $x_0y_3,x_2y_3\in E(G)$. In particular, $y_3\neq y_1$ and $y_3\neq y_0$
  since $x_2y_1\notin E(G)$ and $x_0y_0\notin E(G)$, respectively. Without
  loss of generality, we may now assume that $y_3$ has maximal degree among
  all vertices of $G$ that are adjacent to both $x_0$ and $x_2$.

  If $y_3x_1\notin E(G)$, then $\{x_2,y_3,x_0,y_1,x_1,y_0\}$ induces either
  a $P_6$ or $C_6$ in $G$ depending on whether $x_2$ and $y_0$ are adjacent
  or not. This contradicts our hypothesis. Hence, $y_3$ must be adjacent to
  $x_1$, and thus $y_3\neq y_2$. The same argument with $y_3$ replacing
  $y_2$ yields $x_2y_2 \notin E(G)$. Moreover, $x_2y_0\notin E(G)$, as
  otherwise, $y_2x_0y_1x_1y_0x_2$ is a $P_6$ in $G$. See
  Figure~\ref{fig:fpr}~(b) for the subgraph of $G$ induced by
  $\{x_0,x_1,x_2,y_0,y_1,y_2,y_3\}$.

  Since $G$ contains no heart-vertex, there exists a vertex $x_3$ such that
  $x_3y_3 \notin E(G)$. In particular, $x_3$ must be distinct from
  $x_0,x_1$ and $x_2$. If $x_3$ is adjacent to $y_0$ and not to $y_2$, then
  $y_2x_0y_3x_1y_0x_3$ is a $P_6$. If $x_3$ is adjacent to $y_2$ and not to
  $y_0$, then $x_3y_2x_0y_3x_1y_0$ is a $P_6$. Finally, if $x_3$ is
  adjacent to both $y_0$ and $y_2$, then the set
  $\{y_2,x_0,y_3,x_1,y_0,x_3\}$ induces a $C_6$ in $G$. In summary, $x_3$
  is adjacent to neither $y_0$ nor $y_2$. If $x_3$ is adjacent to $y_1$,
  then $\{x_0,x_1,x_2,x_3,y_0,y_1,y_2,y_3\}$ induces a Sunlet$_4$ in $G$.
  Therefore $x_3$ is not adjacent to $y_1$.

  From Lemma~\ref{lm:maxpr} and the maximality of the degree of $x_0$,
  there exists a vertex $y_4$ such that $\sigma(y_4) \neq \sigma(x_3)$ and
  $x_0y_4,x_3y_4 \in E(G)$. In view of the preceding arguments, $y_4$ is
  distinct from $y_0, y_1, y_2,$ and $y_3$. We show that $y_4x_1,y_4x_2\in
  E(G)$. Indeed, if $x_1y_4 \notin E(G)$, then $x_3y_4x_0y_1x_1y_0$ is a
  $P_6$ of $G$. Hence, $x_1y_4 \in E(G)$. Moreover, if $x_2y_4 \notin
  E(G)$, then $\{x_0,y_2,y_4,x_3,x_1,y_0,y_3,x_2\}$ induces a Sunlet$_4$ in
  $G$. Hence, $x_2y_4 \in E(G)$. Figure~\ref{fig:fpr}~(c) illustrates the
  subgraph of $G$ induced by $\{x_0,x_1,x_2,x_3,y_0,y_1,y_2,y_3,y_4\}$.

  Next, recall that by choice of $y_3$, $\deg_G(y_3) \geq \deg_G(y_4)$ must
  hold. Since $x_3y_4 \in E(G)$ and $x_3y_2 \notin E(G)$, this means that
  there exists $x_4 \in G$ with $\sigma(x_4) \neq \sigma(y_2)$ such that
  $x_4y_2 \in E(G)$ and $x_4y_4 \notin E(G)$. In particular, $x_4$ is
  distinct from $x_0,x_1,x_2$ and $x_3$, since all these vertices are
  adjacent to $y_4$ in $G$.

  Now, if $x_4y_0 \in E(G)$, then $x_3y_4x_2y_3x_4y_0$ is an induced $P_6$
  of $G$. Similarly, if $x_4y_2 \in E(G)$, then $x_3y_4x_2y_3x_4y_2$ is an
  induced $P_6$ of $G$. So, $x_4$ is adjacent in $G$ to neither $y_0$ nor
  $y_2$. The set $\{x_0,y_2,y_3,x_4,x_1,y_0,y_4,x_3\}$ therefore induces a
  Sunlet$_4$ in $G$.  Since the assumptions on $G$ exclude all these cases,
  $G$ must contain a heart-vertex.
\end{proof}

 \begin{figure}
   \centering
   \begin{tikzpicture}[scale=1.2]
     \draw (0.5,-0.5) node{\color{white} $\bullet$} node{$\circ$} node[left]{$y_2$} -- (1,-1) node{$\bullet$} node[left]{$x_0$} -- (1,-2) node{\color{white} $\bullet$} node{$\circ$} node[left]{$y_1$} -- (2,-2) node{$\bullet$} node[right]{$x_1$} -- (2.5,-2.5) node{\color{white} $\bullet$} node{$\circ$} node[right]{$y_0$};
     \draw (1.5,-3) node{(a)};

     \begin{scope}[xshift=3.5cm]
     \draw (0.5,-0.5) node{\color{white} $\bullet$} node{$\circ$} node[left]{$y_2$} -- (1,-1) node{$\bullet$} node[left]{$x_0$} -- (1,-2) node{\color{white} $\bullet$} node{$\circ$} node[left]{$y_1$} -- (2,-2) node{$\bullet$} node[right]{$x_1$} -- (2.5,-2.5) node{\color{white} $\bullet$} node{$\circ$} node[right]{$y_0$};
     \draw (2,-2) -- (2,-1) -- (1,-1);
     \draw (2,-1) node{\color{white} $\bullet$} node{$\circ$} node[right]{$y_3$} -- (2.5,-0.5) node{$\bullet$} node[right]{$x_2$};
     \draw (1.5,-3) node{(b)};
     \end{scope}

     \begin{scope}[xshift=7cm]
     \draw (0.5,-0.5) node{\color{white} $\bullet$} node{$\circ$} node[left]{$y_2$} -- (1,-1) node{$\bullet$} node[left]{$x_0$} -- (1,-2) node{\color{white} $\bullet$} node{$\circ$} node[left]{$y_1$} -- (2,-2) node{$\bullet$} node[right]{$x_1$} -- (2.5,-2.5) node{\color{white} $\bullet$} node{$\circ$} node[right]{$y_0$};
     \draw (2,-2) -- (2,-1) -- (1,-1);
     \draw (2,-1) node{\color{white} $\bullet$} node{$\circ$} node[right]{$y_3$} -- (2.5,-0.5) node{$\bullet$} node[right]{$x_2$} -- (3,0) -- (3.5,0.5) node{$\bullet$} node[right]{$x_3$};
     \draw (1,-1) -- (3,0) node{\color{white} $\bullet$} node{$\circ$} node[right]{$y_4$} -- (2,-2);
     \draw (1.5,-3) node{(c)};
     \end{scope}
     \end{tikzpicture}
     \caption{Three intermediate steps of the proof of Proposition~\ref{pr:heartless} - see text for details. Note that in all stages, the depicted graph is an induced subgraph of $G$.}
     \label{fig:fpr}
 \end{figure}
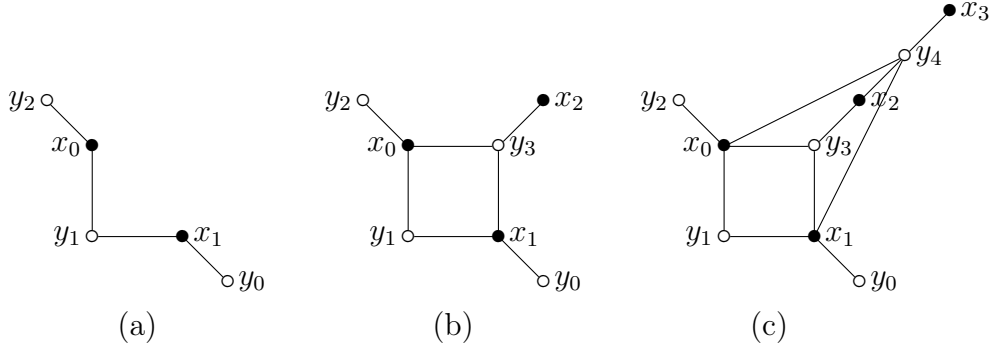

The latter result suggests a relationship with the class of \emph{bi-cographs}(bi-complement reducible graphs)~\cite{Giakoumakis:1997}, which is characterized by the three forbidden subgraphs: $P_7$, the star tree $S_{1,2,3}$ with branches of length $1$,$2$, and $3$, and $\mathrm{Sunlet}_4$. 
Together with Theorem~\ref{th:charh}, we obtain the following
characterization result, which partially resembles the characterization
result~\cite[Proposition 1]{Takaoka:23} for chordal bipartite
graphs. 2qBMGs as the
  $(P6,C_6)$-free bi-cographs. 

\begin{theorem}\label{th:5eq}
  Let $G$ be a graph. The following properties are equivalent.
\begin{itemize}
\item[(i)] $G$ is a un2qBMG.
\item[(ii)] $G$ is bipartite and all connected induced subgraphs of $G$
  contain a heart-vertex.
\item[(iii)] $G$ is bipartite and $(P_6, C_6,\text{Sunlet}_4)$-free.
\item[(iv)]$G$ is $(C_3, C_5, C_6,P_6,\text{Sunlet}_4)$-free.
\item[(v)] $G$ is chordal bipartite and $(P_6,\text{Sunlet}_4)$-free.
\item[(vi)] $G$ is a $(P_6,C_6)$-free bi-cograph.
\end{itemize}
\end{theorem}
\begin{proof}
  The equivalence between (i) and (ii) is stated in Theorem~\ref{th:charh}.
  \cite[Proposition 1]{Takaoka:23} yields the equivalence
  between (iii), (iv), and (v), which is straightforward from the
  definition of chordal bipartite graphs and the characterization of
  bipartite graphs as free of induced odd-cycles.  
  
  We now proceed to show that (ii) and (iii) are equivalent. Assume first that (ii) holds. Then
  $G$ is $(P_6, C_6,\text{Sunlet}_4)$-free, because $P_6$, $C_6$ and
  $\text{Sunlet}_4$ contain no heart-vertex. Hence, (iii) is
  satisfied. Conversely, assume that (iii) holds. Let $H$ be a connected
  induced subgraph of $G$. Clearly, $H$ is bipartite and $(P_6,
  C_6,\text{Sunlet}_4)$-free. Applying Proposition~\ref{pr:heartless} to
  $H$ implies that $H$ has a heart-vertex. Hence, (ii) is satisfied. 
  
  Finally, we established the equivalence between (iii) and (vi). Assume first that (iii) holds. Recall that bi-cographs contain no $P_7, \,S_{1,2,3}$ and $\mathrm{Sunlet}_4$ as induced subgraphs. Hence, (vi) is satisfied, as $P_7$ and $S_{1,2,3}$ contain an induced $P_6$. Conversely, assume that (vi) is true. Clearly, (iii) is verified as bi-cographs are $\mathrm{Sunlet}_4$-free.
\end{proof}

Given that both bi-cographs~\cite{Giakoumakis:2003} and $(P_6,C_6)$-free graphs~\cite{quaddoura2024bipartite} admit linear-time recognition algorithms, Theorem~\ref{th:5eq} thus implies the following result.
  
\begin{corollary}
    un2qBMGs can be recognized in linear time.
\end{corollary}

\section{Concluding Remarks}
In summary, we have provided a comprehensive characterization of undirected
2-quasi-best match graphs (un2qBMGs). Specifically, we established (i) a
structural characterization of un2qBMGs by means of labeled trees
$(T,\sigma,u)$; (ii) a cubic-time recognition algorithm that relies on the
detection of heart-vertices; (iii) an equivalent characterization as the
bipartite graphs in which all connected induced subgraphs have a
heart-vertex; and (iv) a forbidden subgraph characterization description as
the $(P_6,C_6,\mathrm{Sunlet}_4)$-free bipartite graphs. Furthermore, we
showed that un2qBMGs coincide with the $(P_6,C_6)$-free bi-cographs, which
implies the existence of a linear-time recognition algorithm for this
class. Although Algorithm~\ref{alg:main} requires $O(|V(G)|^3)$ time, it is
nevertheless interesting in practice because it explicitly constructs for
each un2qBMG a least-resolved tree $(T,\sigma,u)$.

An interesting direction for future work is to explore whether the linear
decomposition-based recognition algorithms for bi-cographs and
$(P_6,C_6)$-free graphs can be adapted, following the approach
of~\cite{quaddoura2024bipartite}, to recognize un2qBMGs directly. Since
this approach is based on building an (inner node-labeled) decomposition
tree, it would also be interesting to investigate if the resulting
decomposition tree can be used to construct the least-resolved explaining
the tree more efficiently. Such an approach has been successfully used to
build the phylogenetic tree explaining a Fitch graph from di-cographs in
linear time in~\cite{Gei__2018}.

The family of un2qBMGs properly contains several known graph classes. For
example, the bi-quasi-threshold graphs coincide with the domino-free
un2qBMGs~\cite[Theorem~1]{Alecu:2020}. Moreover, according to the online
database \texttt{https://www.graphclasses.org}, the bipartite bi-threshold
graphs~\cite{Hammer:93} form a proper subclass of the un2qBMGs.

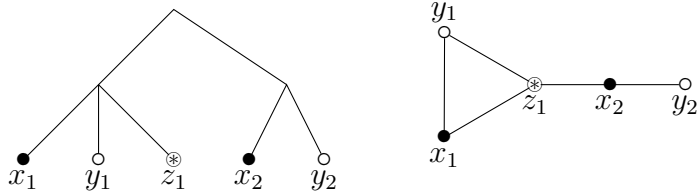
\begin{figure}
  \centering
  \begin{tikzpicture}
    \draw (0,0) node{$\bullet$} node[below]{$x_1$} -- (1,1) -- (2,0) node{\color{white} $\bullet$} node[scale=0.7]{$\circledast$} node[below]{$z_1$};
    \draw (1,0) node{\color{white} $\bullet$} node{$\circ$} node[below]{$y_1$} -- (1,1);
    \draw (3,0) node{$\bullet$} node[below]{$x_2$} -- (3.5,1) -- (4,0) node{\color{white} $\bullet$} node{$\circ$} node[below]{$y_2$};
    \draw (1,1) -- (2,2) -- (3.5,1);

    \begin{scope}[xshift=6cm, yshift=1cm]
    \draw (0:0.8)  node[below]{$z_1$} -- (120:0.8) node{\color{white} $\bullet$} node{$\circ$} node[above]{$y_1$} -- (240:0.81) node{$\bullet$} node[below]{$x_1$} -- cycle;
    \draw (0:0.8) node{\color{white} $\bullet$} node[scale=0.7]{$\circledast$} -- ++(1,0) node{$\bullet$} node[below]{$x_2$} -- ++(1,0) node{\color{white} $\bullet$} node{$\circ$} node[below]{$y_2$};
    \end{scope}
  \end{tikzpicture}
  \caption{Undirected underlying graph of qBMG on three colors, which has no
    heart-vertices.}
  \label{fig:enter-label}
\end{figure}

Clique-width is an important parameter in algorithmic graph theory
since many problems that are NP-hard in general, such a graph
  coloring, become tractable when restricted to graphs of bounded
clique-width~\cite{Alecu:2020}.
Since bi-cographs are known to have bounded clique-width
\cite{Alecu:2020}, un2qBMGs also have bounded clique-width as an
  immediate consequence of Theorem~\ref{th:5eq}(vi). An interesting for
  future work, thus, is to explore the computational complexity of problems
  that remain NP-hard on bi-cographs and $(P_6,C_6)$-free graphs, but may
  become tractable by polynomial-time algorithms when restricted to the
  class of un2qBMGs.
 
Although our results provide a comprehensive understanding of undirected
2-quasi-best match graphs; they do not directly extend to cases with more
than two colors. In particular, the example in Figure~\ref{fig:enter-label}
shows that there is no natural extension of the notion of a heart-vertex to
multipartite graphs such that Lemma~\ref{lm:vheart} holds for three
colors. Consequently, a generalization of Theorem~\ref{th:charh} to three
or more colors appears to be unlikely. Nevertheless, our results yield an
easily testable necessary condition for more than two colors, as every
subgraph of an underlying induced subgraph of a qBMG (unqBMG) induced by a
pair of color classes must be a un2qBMG. Since the least-resolved
explaining trees of a qBMG~\cite[Figure 4]{korchmaros2023quasi}, and thus
also of an unqBMG, are not necessarily unique, it remains an open problem
for future research whether, and if so, how explaining trees for the
un2qBMGs of pairs of color classes can be combined to a labeled tree that
explains an unqBMG. The recognition problem for unqBMGs with three or more
colors, therefore, remains open.

\section*{Acknowledgements}
\noindent
This work was supported in part by the German Research Foundation (DFG,
grant \# 432974470).  Research in the Stadler lab is supported by the
German Federal Ministry of Education and Research (BMBF) through the DAAD
project~57616814 (SECAI, School of Embedded Composite AI) and, jointly with
Tourismus, within the program \emph{Center of Excellence for AI Research
``Center for Scalable Data Analytics and Artificial Intelligence
Dresden/Leipzig''} (ScaDS.AI).

\bibliographystyle{plain}
\bibliography{refs}

\end{document}